\documentclass[oneside,english]{amsart}

\usepackage{amssymb}
\usepackage{amscd}
\usepackage{amsfonts, amsmath, amssymb, graphicx, epstopdf,amstext}
\usepackage{amsthm}
\usepackage{chngcntr}
\usepackage{etoolbox}
\usepackage{lipsum}
\numberwithin{equation}{section} 
\numberwithin{figure}{section} 
\theoremstyle{plain}
\newtheorem*{thm*}{Theorem}
\theoremstyle{plain}
\newtheorem{thm}{Theorem}[section]
\theoremstyle{definition}
\newtheorem{defn}[thm]{Definition}
\theoremstyle{plain}
\newtheorem{lem}[thm]{Lemma}
\theoremstyle{plain}
\newtheorem{prop}[thm]{Proposition}
\theoremstyle{plain}

\theoremstyle{remark}

\theoremstyle{remark}

\theoremstyle{remark}

\newtheorem*{acknowledgement*}{Acknowledgement}

\begin{document}

\title[An Algebraic Characterisation for Finsler Metrics of Constant Flag Curvature]{An Algebraic
  Characterisation for Finsler Metrics of Constant Flag Curvature}

\author[Bucataru]{Ioan Bucataru}
\address{Faculty of  Mathematics \\ Alexandru Ioan Cuza University \\ Ia\c si, 
  Romania}
\email{bucataru@uaic.ro}
\urladdr{http://www.math.uaic.ro/\textasciitilde{}bucataru/}
\author[Fodor]{Dan Gregorian Fodor}
\address{Faculty of  Mathematics \\ Alexandru Ioan Cuza University \\ Ia\c si, 
  Romania}
\email{dan.fodor52@yahoo.com}

\date{\today}

\begin{abstract}
In this paper we prove that a Finsler metrics has
constant flag curvature if and only if the curvature of the induced
nonlinear connection satisfies an algebraic identity with
respect to some arbitrary second rank tensors. Such algebraic
identity appears as an obstruction to the formal integrability of some
operators in Finsler geometry, \cite{BM11, GM00}. This algebraic characterisation for
Finsler metrics of constant flag curvature allows to provide yet another proof for the Finslerian version of
Beltrami's Theorem, \cite{BC18, BC19}. 
\end{abstract}



\maketitle

\section{Introduction}
Riemannian metrics of constant curvature are well understood and
classified. However, one still can find new characterisations
that shed new light on this topic. In this work we consider a 
connected manifold, with $\dim M >2$. 
In \cite{MM19}, the following characterisation is proposed.

\textit{A manifold is a constant curvature manifold if and only if
  \begin{eqnarray} b_{im}R^m_{ljk}+
    b_{km}R^m_{lij}+b_{jm}R^m_{lki}=0, \label{eq:acr}\end{eqnarray} for all symmetric
  tensors $b$.}

The geometric motivation of this characterisation is as follows,
\cite{MM12}. If $X$ and $Y$ are two eigenvectors of a symmetric
tensor $b$ that satisfies \eqref{eq:acr}, then the curvature operator
preserves the subspace generated by he bivector $X\wedge Y$. Since at any point $p\in
M$, the sectional curvature, in the direction of the plane
$P=\operatorname{span}\{X,Y\}\subset T_pM$,  can be viewed as
\begin{eqnarray*}
\kappa_p ( X\wedge Y)  =\frac{R_p(X,Y,X,Y)}{\|X\|_p^2\|Y\|_p^2-<X,Y>_p^2} 
                       = \frac{<X\wedge Y, R(X\wedge Y)>_p}{\|X\wedge Y\|_p^2},
\end{eqnarray*}
 it follows that for such eigenvectors, $\kappa$ does not depend on
 $X\wedge Y$. Having enough symmetric tensors satisfying
 \eqref{eq:acr}, we obtain that $\kappa$ is independent of any
 $2$-plane and using Schur's lemma we have that the sectional curvature is
 constant.

This algebraic characterisation for Riemannian metrics of constant
curvature can be used to provide a very simple and direct proof of
Beltrami's Theorem. If two Riemannian metrics are projectively
equivalent, then their curvature tensors are related by 
\begin{eqnarray*}
\widetilde{R}^h_{ijk} = R^h_{ijk} + \psi_{ij}\delta^h_k -
  \psi_{ik}\delta^h_j, 
\end{eqnarray*} for some symmetric tensor $\psi_{ij}$.
Since $\psi_{ij}\delta^h_k - \psi_{ik}\delta^h_j$ automatically satisfies
\eqref{eq:acr} it follows that $\widetilde{R}$ satisfies
\eqref{eq:acr} if and only $R$ satisfies \eqref{eq:acr}.

In this paper we provide a similar algebraic characterisation for Finsler metrics of
constant flag curvature. For a manifold $M$, we consider $(TM, \pi, M)$ its
tangent bundle and denote by $T_0M=TM\setminus \{0\}$, the tangent space with
the zero section removed. We note that the second iterated tangent
bundle has two vector bundle structures over $TM$, $(TTM, \tau, TM)$
and $(TTM, D\pi, TM)$. 

In Finsler geometry, most of the geometric structures live either on
$TM$ or $T_0M$. The tangent structure (or vertical endomorphism) is
the $(1,1)$-type tensor field $J$ on $TM$ defined  as
\begin{eqnarray*}
J(u)=\left(\tau(u)+ t D\pi(u)\right)'(0), \forall u \in TM.
\end{eqnarray*}

For a Finsler metric $F$, we denote by $S\in {\mathfrak X}(T_0M)$ its geodesic spray. We
consider the geometric setting induced by $S$, \cite{Grifone72, SLK14}, with $h$ the horizontal
projector on $T_0M$ and 
\begin{eqnarray*}
R=\frac{1}{2}[h,h], \end{eqnarray*}
 the curvature (Fr\"olicker-Nijenhuis) tensor of the horizontal
 distribution. 

The vector valued, semi-basic $2$-form $R$ induces an algebraic
derivation $i_R$ of degree $1$. We are interested on its action on the
space of $2$-forms, given by $i_R: \Lambda^2(T_0M) \to
\Lambda^3(T_0M)$, 
\begin{eqnarray*}
i_R\omega(X,Y,Z)=\omega(R(X,Y),Z) + \omega(R(Z,X),Y) +
  \omega(R(Y,Z),X), \forall X,Y,Z \in {\mathfrak X}(T_0M).
\end{eqnarray*} 

The Hessian of the energy of a Finsler metric $F$ gives a symmetric,
second rank positive-definite tensor that can be used to define scalar
products on any tensor space on $T_0M$.  

The flag curvature of a Finsler metric $F$, in the direction of the
flagpole $y$ and the tangent plane $P=\operatorname{span}\{y,X\}\subset T_xM$, can be defined as 
\begin{eqnarray}
\kappa_{(x,y)} (y\wedge X)  = \frac{<y\wedge X, y\wedge R(y\wedge
  X)>_{(x,y)}} {\|y\wedge X\|_{(x,y)}^2}. \label{eq:fc}
\end{eqnarray}  
A Finsler metric has \textit{scalar flag curvature} if the flag curvature
$\kappa$ does not depend on the flag $P$ and it has \textit{constant flag
curvature} if the function $\kappa$ is constant. In view of formula \eqref{eq:fc}, we can see that a Finsler metric has
scalar (constant) flag curvature if and only if $y\wedge R(y\wedge
X)=\kappa y\wedge X$, for some function (constant) $\kappa$ and any flag
$P=\operatorname{span}\{y,X\}$. 

In Finsler geometry there are various characterisations for metrics of constant flag
curvature using some Weyl-type curvature tensors \cite{AZ06, BC18, SM86}. In this
paper we will prove the following algebraic characterisation for
Finsler metrics of constant flag curvature.
\begin{thm} \label{thm:acf}
A Finsler metric is of constant flag curvature if and only if 
\begin{eqnarray}
i_R\omega=0,  \quad \forall \omega \in \Lambda^2(T_0M), \textrm{ satisfying
  \ } i_J\omega=0.\label{eq:acf}
\end{eqnarray}
\end{thm}
As we will see in the proof of Theorem \ref{thm:acf}, the condition \eqref{eq:acf} can be written locally as 
\begin{eqnarray*}
b_{im}R^m_{jk} + b_{km}R^m_{ij} + b_{jm}R^m_{ki}=0, \quad \forall
  b_{ij} \textrm{ symmetric}. \end{eqnarray*}  

The condition \eqref{eq:acf} appears as a first order obstruction to
the formal integrability of the projective metrizability problem
\cite[Theorem 4.3]{BM11}, and as a second order obstruction to
the formal integrability of the Euler-Lagrange operator \cite[\S 5.2]{GM00}.  

We will use the algebraic characterisation from Theorem \ref{thm:acf} to provide a new proof for the
Finslerian version of Beltrami's theorem studied in \cite{BC18, BC19}.

\section{Finsler Metrics and their curvature tensors}

In this work, we consider $M$ a smooth, $n$-dimensional, connected
manifold, with $n>2$. We denote by $(TM, \pi, M)$ the tangent bundle, while
$T_0M=TM\setminus \{0\}$ denotes the tangent space with the zero section
removed. We will use $(x^i)$ to denote local coordinates on $M$ and
$(x^i, y^i)$ for the induced local coordinates on $TM$. 

The canonical submersion $\pi$ induces a regular, $n$-dimensional,
integrable distribution $VTM=\operatorname{Ker}(D\pi)$, which is called the
\textit{vertical distribution}. There is a canonical vertical vector field
${\mathcal C}=y^i\partial/\partial y^i$, called the \textit{Liouville
  vector field}, or the dilation vector field.  The vertical endomorphism has the following local expression
$J=dx^i\otimes \partial/\partial y^i$.    

In this work, we use the Fr\"olicker-Nijenhuis theory of derivations
as it is developed in \cite[Ch. 2]{GM00}. For a vector valued
$k$-form $K$, we denote by $i_K$ the $i_*$-derivation of degree $k-1$, 
and by $d_K$ the $d_*$-derivation of degree $k$. For two vector valued $k$
and $l$-forms $K$ and $L$, we consider the Fr\"olicker-Nijenhuis
bracket $[K,L]$, which is a vector valued $(k+l)$-form.

\begin{defn} \label{dfn:fm} A Finsler metric is a continuous, positive
  function $F:TM \to {\mathbb R}$ that satisfies:
\begin{itemize}
\item[i)] $F$ is smooth on $T_0M$;
\item[ii)] $F$ is positively homogeneous in the fiber coordinates:
  $F(x,\lambda y)=\lambda F(x,y)$, $\forall \lambda >0$;
\item [iii)] the Hessian of the energy function: 
\begin{eqnarray}
g_{ij}(x,y)=\frac{1}{2}\frac{\partial^2F^2}{\partial y^i\partial
  y^j}(x,y) \label{eq:gij}
\end{eqnarray}
is non-degenerate.
\end{itemize}
\end{defn}
The homogeneity condition ii) of a Finsler metric implies, using
Euler's Theorem, that
\begin{eqnarray*}
F^2(x,y)=g_{ij}(x,y)y^iy^j, \quad g_{ij}y^j=\frac{1}{2}\frac{\partial
  F^2}{\partial y^i}.
\end{eqnarray*}
The regularity condition iii) from the definition of a Finsler metric
assures that $dd_JF^2$ is a symplectic structure on $T_0M$. Therefore,
there is a unique vector field $S\in {\mathfrak X}(T_0M)$, satisfying 
\begin{eqnarray}
i_Sdd_JF^2 = -dF^2. \label{eq:ges} \end{eqnarray} 
$S$ is called the \textit{geodesic spray} of the Finsler metric, and it is
given locally by
\begin{eqnarray*}
S=y^i\frac{\partial}{\partial x^i}-2G^i\frac{\partial}{\partial y^i}.
\end{eqnarray*}
 To the geodesic spray we associate the horizontal projector, \cite{Grifone72}, 
\begin{eqnarray}
h=\frac{1}{2}\left(\operatorname{Id}-
  [S,J]\right)=\left(\frac{\partial}{\partial x^i}
  -N^j_i(x,y)\frac{\partial}{\partial y^j}\right)\otimes dx^i, \quad
  N^i_j=\frac{\partial G^i}{\partial y^j}. \label{eq:h}
\end{eqnarray}
The equation \eqref{eq:ges} that uniquely gives the geodesic spray $S$
of a Finsler metric $F$ is equivalent to the following equation:
\begin{eqnarray}
d_hF^2=0. \label{eq:dhf}
\end{eqnarray}
The image of the horizontal projector $h$, $HTM$, is a regular
$n$-dimensional distribution that is supplementary to the vertical
distribution $VTM$. The obstruction to the integrability of the horizontal
distribution is given by the curvature tensor:
\begin{eqnarray*} 
R=\frac{1}{2}[h,h]=R^i_{jk}(x,y)\frac{\partial}{\partial y^i}\otimes
  dx^j\wedge dx^k.
\end{eqnarray*}
In this work we will use the following properties of the curvature tensor $R$:
\begin{itemize}
\item[$R_1)$] Vector-valued, semi-basic $2$-form:
  $R(X,Y)=-R(Y,X)=R(hX, hY)$, $\forall X, Y \in {\mathfrak X}(T_0M)$;
\item[$R_2)$] Satisfies first Bianchi identity: $[J,R]=0$; 
\item[$R_3)$] Satisfies second Bianchi identity: $[h,R]=0$. 
\end{itemize}
In Finsler geometry, there are some other useful curvature
tensors. One is the Jacobi endomorphism, $\Phi=v\circ [S,h]$, connected to the curvature
tensor $R$ by the following formulae:
\begin{eqnarray*}
 \Phi=i_SR, \quad 3R=[J, \Phi].
\end{eqnarray*}

 Similar to the notion of sectional curvature from Riemannian
 geometry, in Finsler geometry we have the concept of flag
 curvature. For $(x,y)\in T_0M$, consider the $2$-dimensional plane $P\subset T_xM$,
 $P=\operatorname{span}\{y, X\}$. The \textit{flag curvature} of the flag
 $\{P, y\}$ can be defined as, \cite{Sh01}, 
\begin{eqnarray}
\kappa_{(x,y)}(y\wedge X)=\frac{X^ig_{il}R^l_{jk}y^jX^k}{g_{ij}y^iy^jg_{ij}X^iX^j
  - (g_{ij}y^iX^j)^2} = \frac{<y\wedge X, y\wedge R(y\wedge
  X)>_{(x,y)}} {\|y\wedge X\|_{(x,y)}^2}. \label{eq:fc1}
 \end{eqnarray}
We say that a Finsler metric $F$ has \textit{scalar flag curvature} (SFC) if
the flag curvature $\kappa=\kappa(x,y)$ does not depend on the flag
$P$, and it has \textit{constant flag curvature} (CFC)  if
the flag curvature $\kappa$ is a constant. In recent years many
geometers obtained new characterisations for Finsler metrics of constant flag
curvature, \cite{AZ06, BC18, SM86}, using some Weyl-type curvature tensors. 

\section{Proof of Theorem \ref{thm:acf}}

In this section, we will provide the proof of Theorem
\ref{thm:acf} using the following lemma that
characterises Finsler metrics of constant flag curvature.

\begin{lem} \label{lem:xicfc} A Finsler metric has constant flag curvature
  if and only if there exists a semi-basic $1$-form $\xi$ such that 
\begin{eqnarray}
R=\xi\wedge J. \label{eq:xicfc}
\end{eqnarray}
\end{lem}
\begin{proof}
Assume that the Finsler metric $F$ has constant flag curvature
$\kappa$. Then, we can rewrite formula \eqref{eq:fc1} as follows:
\begin{eqnarray*}
g_{il}R^l_{jk}y^jX^iX^k = \kappa \left( F^2g_{ik} - g_{il}y^l
  g_{jk}y^j\right) X^iX^k, \forall X^i, X^k.
\end{eqnarray*}
Above formula implies 
\begin{eqnarray}
g_{il}R^l_{jk}y^j = \kappa \left( F^2g_{ik} - g_{il}y^l
  g_{jk}y^j\right), \label{eq:fc2}
\end{eqnarray}
since both sides are symmetric second rank tensors. We denote by
$R^l_{0k}:=R^l_{jk}y^j$, the components of the Jacobi endomorphism
$\Phi=i_SR$. In terms of the Jacobi endomorphism, formula
\eqref{eq:fc2} can be expressed as follows
\begin{eqnarray}
R^l_{0k}=\kappa\left(F^2\delta^l_k - g_{jk}y^j y^l\right). \label{eq:cfcphi}
\end{eqnarray}
Using the fact that one can recover the curvature
tensor from the Jacobi endomorphism, $R=[J,\Phi]/3$, we get from
formula \eqref{eq:cfcphi}, the following expression for the curvature
tensor:
\begin{eqnarray*}
R^l_{jk}=\frac{1}{3}\left(\frac{\partial R^l_{0k}}{\partial y^j}
  - \frac{\partial R^l_{0j}}{\partial y^k}  \right) = \kappa \left(g_{sj}y^s
  \delta^l_k - g_{sk}y^s \delta^j_l\right). 
\end{eqnarray*}
If we consider the semi-basic $1$-form $\xi=\kappa d_JF^2/2 =\kappa
Fd_JF = \kappa g_{is}y^s dx^i$, then the curvature tensor $R$ is given
by formula \eqref{eq:xicfc}.

We assume now that for the curvature tensor $R$ there is a semi-basic
$1$-form $\xi$ that satisfies formula \eqref{eq:xicfc}. Using formula
\eqref{eq:dhf}, as well as the form of the curvature tensor, we obtain 
\begin{eqnarray*}
0=d^2_hF^2=d_RF^2=d_{\xi\wedge J}F^2=\xi\wedge d_JF^2.
\end{eqnarray*}  
From the last formula we obtain that the semi-basic $1$-forms $\xi$
and $d_JF^2$ are proportional and hence there exists $\kappa \in
C^{\infty}(T_0M)$ such that $\xi=\kappa d_JF^2/2$. Hence, the
curvature tensor is given by 
\begin{eqnarray*} 
R=\frac{\kappa}{2}d_JF^2 \wedge J. 
\end{eqnarray*}
The curvature tensor satisfies the first Bianchi identity
$[J,R]=0$, which implies $d_J\xi \wedge J=0$. We take the trace of
this vector-valued $3$-form to obtain $(n-2)d_J\xi=0$, that gives $d_J\xi=0$, since we made
the assumption that $n>2$. From the expression of the curvature
$1$-form $\xi$, we obtain that $d_J\kappa=0$.

The curvature tensor also satisfies the second Bianchi identity
$[h,R]=0$, which yields $d_h\xi \wedge J=0$. Again we take the trace
to obtain $(n-2)d_h\xi=0$,  which in view of our assumption regarding
the dimension, gives $d_h\xi=0$. Using the expression of the curvature
$1$-form $\xi$, we obtain that $d_h\kappa=0$.

The two conditions $d_J\kappa=0$ and $d_h\kappa=0$ assures that
$\kappa$ is a constant, \cite[Theorem 9.4.11]{SLK14}. For this constant, the curvature tensor $R$
satisfies formula \eqref{eq:fc1} and therefore the Finsler metric has
constant sectional curvature $\kappa$.
\end{proof}

One can view Lemma \ref{lem:xicfc} as a reformulation, in dimension
$n\geq 3$, of the Finslerian version of Schur Lemma
\cite[Theorem 3.2]{BC19}, the expression \eqref{eq:xicfc} of the
curvature tensor contains two information: the geodesic spray is
isotropic and the curvature $1$-form $\xi$ satisfies $d_J\xi=0$. 

\subsection{Proof of Theorem \ref{thm:acf}}

We provide now the proof of Theorem \ref{thm:acf}. The simplest
implication is the sufficiency. We assume that the
Finsler metric has constant flag curvature. Then, using Lemma
\ref{lem:xicfc}, the curvature tensor is given by $R=\xi\wedge J$, for
some semi-basic $1$-form $\xi$. For any $\omega \in \Lambda^2(T_0M)$,
satisfying $i_J\omega=0$, we have 
\begin{eqnarray*}
i_R\omega=i_{\xi\wedge J}\omega = \xi \wedge i_J\omega=0. 
\end{eqnarray*}
   
We assume now that for a Finsler metric $F$, its curvature tensor $R$
satisfies the identity \eqref{eq:acf}, $i_R\omega=0$, $\forall \omega
\in \Lambda^2(T_0M)$ with $i_J\omega=0$. 

We first characterise the $2$-forms $\omega$ satisfying the condition
$i_J\omega=0$. For arbitrary vector fields $X, Y \in {\mathfrak
  X}(T_0M)$, it means 
\begin{eqnarray*}
\omega(JX, Y)+\omega(X, JY)=0.
\end{eqnarray*}
If $Y=JZ$ is a vertical vector field, it follows that $\omega(JX,
JZ)=0$. Therefore $\omega$ vanishes on any pair of vertical vector
fields (the vertical distribution is a Lagrangian distribution for
$\omega$). Consider $\{dx^i, \delta y^i:=dy^i + N^i_j dx^j\}$ a local
basis, adapted to the horizontal and vertical distributions. With
respect to this basis, a $2$-form $\omega$ satisfying $i_J\omega=0$
can be expressed as
\begin{eqnarray*}
\omega = a_{ij}dx^i\wedge dx^j + b_{ij}dx^i \wedge \delta
  y^j. \end{eqnarray*}
Since $i_J\omega = (b_{ij}-b_{ji}) dx^i\wedge dx^j$, the condition $i_J\omega=0$ implies
$b_{ij}=b_{ji}.$ We have now
\begin{eqnarray*}
i_R\omega = \left(b_{im}R^m_{jk} + b_{km}R^m_{ij} + b_{jm}R^m_{ki}
  \right) dx^i\wedge dx^j \wedge dx^k.
\end{eqnarray*} 
Hence, the identity \eqref{eq:acf} is satisfied if and only if 
\begin{eqnarray} 
b_{im}R^m_{jk} + b_{km}R^m_{ij} + b_{jm}R^m_{ki}=0, \label{eq:acf1}
\end{eqnarray}
for any symmetric tensor $b_{ij}$. The identity \eqref{eq:acf1} can be written in the following equivalent form 
\begin{eqnarray*}
b_{sl}\left(\delta^s_iR^l_{jk} + \delta^s_kR^l_{ij} +
  \delta^s_jR^l_{ki}  \right) =0, \quad \forall b_{sl}.
\end{eqnarray*}
The above identity implies that that the $(2,3)$ type tensor
$\delta^s_iR^l_{jk} + \delta^s_kR^l_{ij} + \delta^s_jR^l_{ki}$ is
skew-symmetric in the $2$ contravariant indices. This means 
\begin{eqnarray}
\delta^s_iR^l_{jk} + \delta^s_kR^l_{ij} +
  \delta^s_jR^l_{ki} + \delta^l_iR^s_{jk} + \delta^l_kR^s_{ij} +
  \delta^l_jR^s_{ki} = 0. \label{eq:acf2} 
\end{eqnarray}  
For the $(1,2)$-type curvature tensor $R$, we consider its trace, the
semi-basic $1$-form $(n-1)\xi_k:=R^s_{sk}=-R^s_{ks}$. In \eqref{eq:acf2}, if we take
the trace $i=s$, we obtain:
\begin{eqnarray*}
nR^l_{jk} + R^l_{kj} + R^l_{kj} + R^l_{jk} - (n-1)\delta^l_k\xi_j +
  (n-1) \delta^l_j \xi_k=0.
\end{eqnarray*}
Last formula can be written as 
 \begin{eqnarray*}
R^l_{jk}  = \xi_j \delta^l_k - \xi_k \delta^l_j,
\end{eqnarray*}
which means that the curvature tensor $R$ is given by formula
\eqref{eq:xicfc} and in view of Lemma \ref{lem:xicfc}, we obtain that the Finsler
metric has constant flag curvature. 

The above proof of the Theorem \ref{thm:acf} can be used to provide a
similar result for an arbitrary semi-basic, vector valued $2$-form $K$
on $T_0M$, not necessary the curvature tensor $R$.

\begin{prop} \label{prop1} A semi-basic, vector valued $2$-form $K$ on
  $T_0M$ satisfies the identity
\begin{eqnarray}
i_K\omega=0, \quad \forall \omega\in \Lambda^2(T_0M), \quad 
  i_J\omega=0 \label{iko}
\end{eqnarray}
if and only if there exists a semi-basic $1$-form $\xi \in
\Lambda^1(T_0M)$ such that 
\begin{eqnarray}
K=\xi\wedge J.  \label{kxij}
\end{eqnarray}
The semi-basic $1$-form $\xi=\xi_idx^i$ from
  formula \eqref{kxij}, if it exists, it is unique, being given by $(n-1)\xi_i=K^j_{ji}=-K^j_{ij}.$
\end{prop}

\subsection{A new proof of the Finslerian version of Beltrami's Theorem}

Two Finsler metrics $F$ and $\widetilde{F}$ are projectively related if their geodesics
coincide as unparameterised oriented curves, their geodesic sprays $S$
and $\widetilde{S}$ being related by $\widetilde{S} = S- 2P{\mathcal C}$. The function $P\in C^{\infty}(T_0M)$ is positively $1$-homogeneous in
the fiber coordinates and it is called the \textit{projective factor}.

We will use now the algebraic characterisation for Finsler metrics of
constant flag curvature given by Theorem \ref{thm:acf} to provide
another proof of the Finslerian version of Beltrami's Theorem.

\begin{thm}
Consider $F$ and $\widetilde{F}$ two projectively related Finsler
metrics and asume that one of them is of constant flag
curvature. Then, the other metric is of constant curvature as well if
and only if the projective factor $P$ satisfies the Hamel equation $d_hd_JP=0$.
\end{thm}
\begin{proof}
For two projectively related Finsler metrics $F$ and $\widetilde{F}$,
their curvature tensors $R$ and $\widetilde{R}$ are related by \cite[(4.8)]{BM12}
\begin{eqnarray}
\widetilde{R}=R + \eta\wedge J - d_J\eta \otimes {\mathcal C}, \quad
  \eta=Pd_JP - d_hP. \label{eq:rrt} \end{eqnarray} 
We assume that the Finsler metric $F$ has constant flag
curvature and hence we obtain that the curvature
tensor $R$ is given by $R=\xi \wedge J$. Then
formula \eqref{eq:rrt} can be written as follows
\begin{eqnarray}
\widetilde{R}=\left(\xi + \eta\right)\wedge J - d_J\eta \otimes {\mathcal C}. \label{eq:rrt1} \end{eqnarray} 
Using this form of the curvature tensor $\widetilde{R}$ and Theorem
\ref{thm:acf}, we obtain that $\widetilde{F}$ has
constant curvature if and only if   $d_J\eta=0$. 
Since $d_J\eta=d_hd_JP$, we have that $\widetilde{F}$ has
constant flag curvature if and only if the projective factor $P$ satisfies
Hamel's equation $d_hd_JP=0$.
\end{proof}



\begin{thebibliography}{9}
\bibitem{AZ06}{Akbar-Zadeh, H.}: \emph{Initiation to global Finslerian
  geometry}, Elsevier, North-Holland Mathematical Library, 2006.

\bibitem{BC18} Bucataru I., Cre\c tu G.: \emph{A characterisation for
    Finsler metrics of constant curvature and a Finslerian version of
    Beltrami theorem},  J. Geom. Anal.,  
DOI: 10.1007/s12220-019-00158-7, arXiv:1808.05001v2.

\bibitem{BC19} Bucataru I., Cre\c tu G.: \emph{Finsler spaces of
    constant flag curvature and their projective geometry},
  arXiv:1902.05274.

\bibitem{BM11} Bucataru I., Muzsnay Z.: \emph{Projective Metrizability
    and Formal Integrability}, Symmetry, Integrability and Geometry:
  Methods and Applications, \textbf{7}(2011), 114.

\bibitem{BM12} Bucataru I., Muzsnay Z.: \emph{Projective and Finsler
    metrizability: parameterization-rigidity of geodesics},
  Int. J. Math., \textbf{23} (6) (2012), 1250099.

\bibitem{Grifone72}  Grifone, J.: \emph{Structure presque-tangente et
    connexions I}, Ann. Inst. Fourier, \textbf{22} (1972), 287--334.

\bibitem{GM00} Grifone J., Muzsnay Z.: \emph{Variational Principles For
    Second-Order Differential Equations}, World Scientific, 2000.

\bibitem{MM12} {Mantica, C.A., Molinari, L.G.:} \emph{Extended
    Derdzi\'nski-Shen theorem for curvature tensors}, Colloq. Math.,
  \textbf{28} (1) (2012), 1--6.   

\bibitem{MM19} {Mantica, C.A., Molinari, L.G.:} \emph{The Jordan
    algebras of Riemann, Weyl and curvature compatible tensors}, arXiv:1910.03929. 
 

\bibitem{Sh01}  Shen Z.: \emph{Differential geometry of spray and Finsler spaces}, Springer, 2001

\bibitem{SM86}{Sinha, B.B., Matharoo, A.S.}: \emph{On Finsler spaces
    of constant curvature}, Indian J. Pure Appl. Math.,
  \textbf{17} (1) (1986), 66--73.

\bibitem{SLK14} {Szilasi, J., Lovas, R., Kert\'esz, D.:}
  \emph{Connections, sprays and Finsler structures}, World Scientific, 2014.

\end{thebibliography}
\end{document}